\documentclass[12pt,twoside]{article}
\usepackage{verbatim,amssymb,amsmath}
\usepackage{graphicx,array}
\usepackage{mystyle}

\usepackage[toc]{appendix}

\usepackage{xcolor}
\usepackage{relsize}

\begin{document}

\title{Well-posedness of the two-component Fornberg-Whitham system in Besov spaces}
\author{
{Prerona Dutta} 
\thanks{Department of Mathematics, The Ohio State University: dutta.105@osu.edu.
}
}
\markboth{Dutta}{Fornberg-Whitham System in Besov spaces}

\maketitle

\begin{abstract}
The present paper establishes well-posedness for the two-component Fornberg-Whitham system in Besov spaces. First the existence and uniqueness of its solution is proved, then it is shown that the corresponding data-to-solution map is continuous, provided the initial data belong to Besov spaces.\\


\end{abstract}

\begin{center}  
{\em \today}\end{center}

\section{Introduction}\label{sec1}

In this paper, we consider the following two-component Fornberg-Whitham (FW) system for a fluid
\begin{equation}\label{FW0}
\begin{cases}
u_t + uu_x = \partial_x\left(1-\partial_x^2\right)^{-1}\left(\eta - u\right)~,~~ (t,x)\in\mathbb{R}^{+}\times\mathbb{R}\\
\eta_t + \left(\eta u\right)_x = 0\\
\left(u, \eta\right)(0, x) = \left(u_0, \eta_0\right)(x)~,~~x\in \mathbb{R}
\end{cases}
\end{equation}
where $u=u(x,t)$ is the horizontal velocity of the fluid and $\eta = \eta(x,t)$ is the height of the fluid surface above a horizontal bottom. This system was first proposed in \cite{FTYY}, where the authors studied the bifurcations of its travelling wave solutions. The Fornberg-Whitham (FW) equation
\begin{equation}\label{fw}
u_t - u_{xxt} + u_x + uu_x = uu_{xxx} + 3u_xu_{xx}
\end{equation}
was derived in \cite{FW} as a model to study waves on shallow water surfaces. To obtain this, Fornberg and Whitham considered the integro-differential equation
\begin{equation}\label{wkdv}
u_t + uu_x + \int_{-\infty}^{\infty} K(x-\xi)u_{\xi}(\xi, t)d\xi = 0
\end{equation}
which was proposed by Whitham in \cite{W} as a better model for examining short wave phenomena. The kernel $K(x)$ can be chosen to obtain any required dispersion relation, by viewing it as the Fourier transform of the phase velocity. To investigate the qualitative behavior of wave breaking in the context of water waves, \eqref{fw} was derived in \cite{FW} by choosing $K(x) = \frac{1}{2}\nu e^{-\nu|x|}$, where $\nu$ is an adjustable parameter. 
\medskip

\noindent The FW equation continued to receive increasing attention over the years, as the study of travelling wave solutions became more important in research for several areas of Physics. Motivated by the generation of the two-component Camassa–Holm (CH) equation in \cite{CI} and the two-component Degasperis Procesi equation in \cite{FTZ}, the authors of \cite{FTYY} generalized equation \eqref{fw} to obtain the two-component FW system \eqref{FW0}. The hydrodynamical derivation for the two-component CH system in \cite{CI} required that $u(x,t) \to 0$ and $\eta(x,t) \to 1$ as $|x| \to \infty$ at any instant $t$. This holds for the system \eqref{FW0} as well, since \eqref{FW0} was derived in \cite{FTYY} following the same procedure as in \cite{CI}. The system \eqref{FW0} was further investigated in \cite{FTYY} and the authors provided parametric expressions for its smooth soliton solution, kink solution, antikink solution and uncountably infinite smooth periodic wave solutions.
\medskip

\noindent The well-posedness of the FW equation and its Cauchy problem in Sobolev spaces $H^s$ for $s > \frac{3}{2}$ have been established in \cite{EEP} and \cite{H} respectively. Furthermore, in \cite{HT} Holmes and Thompson found that this equation is well-posed in Besov spaces $B^s_{2,r}$, where $s$ is a Sobolev exponent and $r$ is related to a H\"older exponent, and that the data-to-solution map is not uniformly continuous in this case. They also proved existence and uniqueness of real analytic solutions for this equation and provided a blow-up criterion for solutions.
\medskip

\noindent The FW equation belongs to the family of nonlinear wave equations described by
\begin{equation}\label{nlw}
u_t + u u_x = \mathcal{L} (u, u_x)
\end{equation}
which has been studied extensively in existing literature. For $\mathcal{L}(u) = -\frac{1}{6}\partial_x ^3 u$, \eqref{nlw} becomes the KdV equation, while for $\mathcal{L} (u) = -(I -\partial_x ^2)^{-1}\partial_x(u^2 + \frac{1}{2}u_x^2)$, it is the Camassa-Holm (CH) equation. The two-component system \eqref{FW0} captures several features of surface waves in an incompressible fluid, e.g. nonlocality and wave breaking, which the KdV equation does not. Therefore it is useful to look into properties of the FW system (equation) as an alternative to the KdV equation for water waves. 
\medskip

\noindent Establishing well-posedness of strong solutions for the system \eqref{FW0} in various spaces is a challenging problem. In \cite{XZL}, the authors proved local well-posedness of the FW system with initial data in $H^s \times H^{s-1}$ for $s>\frac{3}{2}$ and presented a blow-up criterion by a local-in-time a priori estimate. They also imposed sufficient conditions on the initial data that may lead to wave breaking and analytically demonstrated the existence of periodic traveling waves using a local bifurcation theorem. 
\medskip

\noindent Our objective is to establish local well-posedness of the two component FW system in Besov spaces $B^s_{p,r} \times B^{s-1}_{p,r}$ for $s > \max\{2+\frac{1}{p}, \frac{5}{2}\}$, where $s$ is a Sobolev exponent, $p$ is an ${\bf{L}}^p$ space exponent and $r$ is related to a H\"older exponent. Besov spaces are an interesting class of functions of growing relevance in the study of nonlinear partial differential equations as they generalize Sobolev spaces and are more effective at measuring regularity properties of functions. To examine the FW system in Besov spaces, we define $\rho = \eta - 1$ and study the following system:
\begin{equation}\label{FW}
\begin{cases}
u_t + uu_x = \partial_x\left(1-\partial_x^2\right)^{-1}\left(\rho - u\right)~,~~ (t,x)\in\mathbb{R}^{+}\times\mathbb{R}\\
\rho_t + u\rho_x + \rho u_x + u_x = 0\\
\left(u, \rho\right)(0, x) = \left(u_0, \rho_0\right)(x)~,~~x\in \mathbb{R}
\end{cases}
\end{equation}
Now, $\rho(x,t) \to 0$ as $|x| \to \infty$. For our problem to be well-posed in the sense of Hadamard, we must show existence and uniqueness of the solution to the system \eqref{FW} and also continuity of the data-to-solution map when initial data belongs to the aforementioned Besov spaces. To prove existence, we consider a sequence of approximate solutions to \eqref{FW} and construct a system of linear transport equations. First we show that the solutions to this system are uniformly bounded on a common lifespan. Then using a compactness argument, we extract a subsequence that converges to a solution of the original system \eqref{FW}. Uniqueness and continuous dependence on initial data follow using an approximation argument similar to that in \cite{HT} which dealt with the single FW equation, with appropriate modifications required upon adding $\rho$ to the model.
\medskip

\noindent This paper is organized as follows. In Section \ref{sec2}, we state important definitions and properties related to Besov spaces and linear transport equations. Section \ref{sec3} begins with the main existence result Theorem \ref{existmain} for the FW system \eqref{FW}. The proof involves finding the minimum lifespan, which is presented in Lemma \ref{lifespan}. Uniqueness of the solution is verified in Proposition \ref{uniqueness}. Finally we prove continuity of the data-to-solution map when the initial data belong to $B^s_{p,r} \times B^{s-1}_{p,r}$ for $s > \max\{2+\frac{1}{p}, \frac{5}{2}\}$, thus establishing local well-posedness for the two-component FW system.

\section{Preliminaries}\label{sec2}

This section is a review of relevant definitions and results on Besov spaces and linear transport equations, from \cite{BCD, D1, D2}, that will be used throughout the rest of this paper. Additionally, it describes some analytical tools used in Section \ref{sec3}. 

\subsection{Besov spaces}

Let $\mathcal{S}(\R)$ denote the Schwartz space of smooth functions on $\mathbb{R}$ whose derivatives of all orders decay at infinity. Then the set $\mathcal{S}'(\R)$ of temperate distributions is the dual set of $\mathcal{S}(\R)$ for the usual pairing.\\

\begin{lemma}[Littlewood-Paley decomposition]
There exists a pair of smooth radial functions $(\chi, \varphi)$ taking values in $[0,1]$ such that $\chi$ is supported in the ball $B = \{\xi \in \R~\big|~ |\xi|\leq \frac{4}{3}\}$ and $\varphi$ is supported in the ring $C = \{\xi \in \R~\big|~  \frac{3}{4}\leq|\xi|\leq \frac{8}{3}\}$. Moreover, for all $\xi \in \R$,
$$\chi(\xi) + \sum\limits_{q\in\mathbb{N}}\varphi(2^{-q}\xi) = 1$$
and
\begin{align*}
&\mathrm{supp}~\varphi (2^{-q}\cdot) \cap \mathrm{supp}~\varphi (2^{-q'}\cdot) = \emptyset~~,~~\text{if}~|q-q'|\geq 2~,\\&\mathrm{supp}~\chi (\cdot) \cap \mathrm{supp}~\varphi (2^{-q}\cdot) = \emptyset~~,~~\text{if}~|q|\geq 1 ~.
\end{align*}
\end{lemma}
\noindent Then for $u\in\mathcal{S}'(\R)$, the nonhomogeneous dyadic intervals are defined as follows:
\begin{align*}
&\Delta_q u = 0~,~ \text{if}~q\leq -2\\
&\Delta_{-1} u = \chi(D)u = \mathcal{F}^{-1} \chi \mathcal{F} u\\
&\Delta_q u = \varphi(2^{-q}D)u =  \mathcal{F}^{-1} \varphi(2^{-q}\xi) \mathcal{F} u~,~ \text{if}~q\geq 0~.
\end{align*}
Thus $u = \sum\limits_{q\in\mathbb{Z}} \Delta_q u$ in $\mathcal{S}'(\R)$.\\
\\
\begin{definition}
The low frequency cut-off $S_q$ is defined by
$$S_q u = \sum\limits_{p=-1}^{q-1} \Delta_p u = \chi(2^{-q}D)u = \mathcal{F}^{-1} \chi(2^{-q}\xi) \mathcal{F} u~,~ \text{for all}~ q \in \mathbb{N}~.$$
\end{definition}
\noindent Then we have
\begin{align*}
\Delta_p\Delta_q u &\equiv 0 ~,~\text{if}~|p-q|\geq 2\\
\Delta_q(S_{p-1}u\Delta_p v) &\equiv 0 ~,~\text{if}~|p-q|\geq 5~,~\text{for all}~ u, v \in \mathcal{S}'(\R)~
\end{align*}
and by Young's inequality it also follows that for all $p \in [1, \infty]$,
$$ \|\Delta_q u\|_{{\bf L}^p} \leq  \|u\|_{{\bf L}^p}~~ \text{and}~~  \|S_q u\|_{{\bf L}^p} \leq M\|u\|_{{\bf L}^p}$$
where $M$ is a positive constant independent of $q$.\\

\noindent Using the Littlewood-Paley decomposition we define Besov spaces as follows.\\
\begin{definition}[Besov spaces]
Let $s\in \R$ and $p$, $r$ $\in [1, \infty]$. Then the Besov spaces of functions are defined as
$$B^s_{p,r} \equiv B^s_{p,r} (\R) = \{u \in \mathcal{S}'(\R)~\big|~ \|u\|_{B^s_{p,r} } < \infty\}~,$$
where
$$
\|u\|_{B^s_{p,r} } = 
\begin{cases}
\left(\sum\limits_{q\geq -1}(2^{sq}\|\Delta_q u\|_{{\bf L}^p})^r\right)^{1/r}~,~ \text{if}~ 1\leq r < \infty\\ 
\sup\limits_{q\geq -1} 2^{sq}\|\Delta_q u\|_{{\bf L}^p}~~~~~~~~~~~~,~ \text{if}~ r = \infty
\end{cases}
~.
$$
\\In particular, $B^{\infty}_{p,r} = \bigcap\limits_{s\in\R} B^s_{p,r}$.\\
\end{definition}
\noindent Following are some important properties proved in \cite[Section 2.8]{BCD} and \cite[Section 1.3]{D2} that facilitate the study of nonlinear partial differential equations in Besov spaces.\\
\begin{lemma}\label{besov}
Let $s, s_j \in \R$ and $1\leq p, r, p_j, r_j \leq \infty$, for $j=1,2$. Then the following hold:
\begin{enumerate}
\item[(1)] Topological property: $B^s_{p,r}$ is a Banach space which is continuously embedded in $\mathcal{S}'(\R)$.
\item[(2)] Density: $C^{\infty}_c$ is dense in $B^s_{p,r}$ if and only if $p$ and $r$ are finite.
\item[(3)] Embedding: $B^s_{p_1,r_1} \hookrightarrow B^{s-(\frac{1}{p_1} - \frac{1}{p_2})}_{p_2,r_2}$, if $p_1\leq p_2$ and $r_1\leq r_2$, and
$$B^{s_2}_{p,r_2} \hookrightarrow B^{s_1}_{p,r_1}~~\text{locally compact if}~ s_1 < s_2~.$$
\item[(4)] Algebraic property: For all $s>0$, $B^s_{p,r} \cap {\bf L}^{\infty}$ is a Banach algebra. $B^s_{p,r}$ is a Banach algebra $\Longleftrightarrow$ $B^s_{p,r} \hookrightarrow {\bf L}^{\infty}$ $\Longleftrightarrow$ $s> \frac{1}{p}$ (or $s\geq \frac{1}{p}$ and $r=1$). In particular, $B^{1/p}_{p,1}$ is continuously embedded in $B^{1/p}_{p,\infty} \cap {\bf L}^{\infty}$ and $B^{1/p}_{p,\infty} \cap {\bf L}^{\infty}$ is a Banach algebra.
\item[(5)] For $s > \max\{2+\frac{1}{p}, \frac{5}{2}\}$,
$$ \|fg\|_{B^{s-1}_{p,r}} \leq M \|f\|_{B^{s-1}_{p,r}}\|g\|_{B^{s}_{p,r}}~,$$
where $M$ is a constant that depends on $s$, $p$ and $r$.
\item[(6)] Fatou property: If the sequence $\{f_n\}_{n\in\mathbb{N}}$ is uniformly bounded in $B^{s}_{p,r}$ and converges weakly in $\mathcal{S}'(\R)$ to $f$, then $f\in B^{s}_{p,r}$ and 
$$\|f\|_{B^{s}_{p,r}} \leq  \kappa \liminf\limits_{n\to\infty}\|f_{n}\|_{B^{s}_{p,r}}~$$
where $\kappa$ is a constant that depends on $s$, $p$ and $r$.
\end{enumerate}
\end{lemma}
\begin{definition}
A smooth function $f:\R\to \R$ is said to be an $S^m$-multiplier if for all multi-index $\alpha$, there exists a constant $C_{\alpha}$ such that for all $\xi \in \R$
$$|\partial^{\alpha} f(\xi)|~\leq~ C_{\alpha}(1+|\xi|)^{m-|\alpha|}~.$$
\end{definition}
\begin{proposition}\label{multiplier}
Let $m\in\R$ and $f$ be an $S^m$-multiplier. Then for all $s\in \R$, $1\leq p$ and $r\leq\infty$, the operator $f(D)$ defined for all $u\in\mathcal{S}'(\R)$ as
$$f(D)u = \mathcal{F}^{-1}f\mathcal{F}u$$
is continuous from $B^s_{p,r}$ to $B^{s-m}_{p,r}$.
\end{proposition}

\subsection{Linear transport equation}

Given a linear transport equation, Proposition A.1 in \cite{D1} proves the following estimate for its solution size in Besov spaces.\\
\begin{proposition}\label{danchinLT}
Consider the linear transport equation
\begin{equation}\label{lineartrans}
\begin{cases}
\partial_t f + v \partial_x f = F \\
f(x,0) = f_0(x)
\end{cases}
\end{equation}
where $f_0 \in B^s_{p,r}$, $F \in {\bf L}^1((0, T); B^{s}_{p,r})$ and $v$ is such that $\partial_x v \in {\bf L}^1((0, T); B^{s-1}_{p,r})$. Suppose $f \in {\bf L}^{\infty}((0, T); B^{s}_{p,r}) \cap C([0,T]; \mathcal{S}')$ is a solution to \eqref{lineartrans}. Let $1\leq p \leq \infty$ and either $s> 1+\frac{1}{p}$, $r\in (1, \infty)$ or $s\geq 1+\frac{1}{p}$, $r=1$. Then for some constant $C$ which depends on $s$ and $p$, and $$V(t) = \int_0^t \|\partial_x v(\tau)\|_{B^{s-1}_{p,r}}~d\tau~,$$ we have
\begin{equation}\label{transestimate}
\|f(t)\|_{B^s_{p,r}}~\leq~ e^{CV(t)}\left(\|f_0\|_{B^s_{p,r}} + C\int_0^t  e^{-CV(\tau)}\|F(\tau)\|_{B^s_{p,r}}~d\tau\right)~.
\end{equation}
Moreover, if $r=\infty$ then $f\in C([0,T], B^{s'}_{p,1})$ for all $s'<s$.
\end{proposition}

\subsection{Miscellaneous}
The following notions have been utilized in the proof of Theorem \ref{existmain}.
\begin{enumerate}
\item[I.] Friedrichs mollifier: Given a non-negative function $\varphi \in C_c^{\infty}(\R)$ such that $\int_{\R} \varphi dx = 1$, for any $\varepsilon>0$ we define a Friedrichs mollifier by
\begin{equation}\label{moll}
J_{\varepsilon} f (x) = \varepsilon^{-1}\int_{\R}\varphi\left(\frac{x-y}{\varepsilon}\right)f(y)dy
\end{equation}
where $f\in{\bf L}^p(\R)$ and $1\leq p\leq\infty$.
\item[II.] Operator $\Lambda$: For our problem, we define $\Lambda = 1-\partial_x^2$ . Then the system \eqref{FW} becomes
\begin{equation}
\begin{cases}
u_t + uu_x = \Lambda^{-1}[\partial_x\left(\rho - u\right)]~,~~ (t,x)\in\mathbb{R}^{+}\times\mathbb{R}\\
\rho_t + u\rho_x + \rho u_x + u_x = 0\\
\left(u, \rho\right)(0, x) = \left(u_0, \rho_0\right)(x)~,~~x\in \mathbb{R}
\end{cases}~.
\end{equation}
$\Lambda^{-1}$ is a well-defined operator whose Fourier transform is $\mathcal{F}\left(\Lambda^{-1}f\right) = \frac{1}{1+\xi^2}\hat{f}(\xi)$ for any test function $f$. Now, $\Lambda^{-1}$ is an $S^2$-multiplier. Therefore by Proposition \ref{multiplier}, for some constant $\theta$ depending on $s$, $p$ and $r$ it holds that
\begin{equation}\label{lambinv}
\|\Lambda^{-1}\partial_x\left(\rho^n - u^n\right)\|_{B^s_{p,r}} ~\leq~ \theta \|\rho^n - u^n\|_{B^{s-1}_{p,r}}~.
\end{equation}

\end{enumerate}

\section{Local Well-posedness}\label{sec3}
In this section, we prove existence and uniqueness of the solution to the FW system \eqref{FW}, and establish continuous dependence of its data-to-solution map in $B^s_{p,r} \times B^{s-1}_{p,r}$.
\subsection{Existence and Lifespan}\label{subsec1}
\begin{theorem}\label{existmain}
Let $s> \max\{2+\frac{1}{p}, \frac{5}{2}\}$, $p \in [1, \infty]$, $r \in [1, \infty)$ and $(u_0, \rho_0) \in B^s_{p,r} \times B^{s-1}_{p,r}$. Then the system \eqref{FW} has a solution $(u, \rho) \in C\left([0,T];B^s_{p,r} \times B^{s-1}_{p,r}\right)$ for 
$$T~=~  \frac{3}{16C\left(\|u_0\|_{B^s_{p,r}} + \|\rho_0\|_{B^{s-1}_{p,r}}\right)^2}~.$$
\end{theorem}
\begin{proof}
Consider the sequences of smooth functions $\{u^n\}_{n\geq 0}$ and $\{\rho^n\}_{n\geq 0}$ with $u^0 = 0$ and $\rho^0 = 0$ that solve the following system of linear transport equations
\begin{equation}\label{transport}
\begin{cases}
u^{n+1}_t + u^nu^{n+1}_x = \Lambda^{-1}[ \partial_x\left(\rho^n - u^n\right)]\\
\rho^{n+1}_t + u^n \rho^{n+1}_x  = -\rho^nu^{n}_x - u^n_x\\
u^{n+1}(x,0) = J_{n+1}u_0 (x)\\
\rho^{n+1}(x,0) = J_{n+1}\rho_0 (x)\\
\end{cases}
\end{equation}
where $J_{n+1}$ is a Friedrichs mollifier as defined in \eqref{moll} and $\Lambda = 1-\partial_x^2$.
\medskip

\noindent First we need to show that solutions to the system \eqref{transport} are uniformly bounded on a common lifespan. Applying Proposition \ref{danchinLT}, for some constants $K_1$ and $K_2$ that depend on $s, p$ and $r$, we have 
\begin{multline}\label{est1}
\|u^{n+1}(t)\|_{B^s_{p,r}}~\leq~ e^{K_1V_n(t)} \|u_0\|_{B^s_{p,r}} ~+~\\
 K_1\int_0^t  e^{K_1V_n(t)-K_1V_n(\tau)}\|\Lambda^{-1}[ \partial_x\left(\rho^n - u^n\right)(\tau)]\|_{B^s_{p,r}}~d\tau
\end{multline}
and 
\begin{multline}\label{est2}
\|\rho^{n+1}(t)\|_{B^{s-1}_{p,r}}~\leq~ e^{K_2V_n(t)}\|\rho_0\|_{B^{s-1}_{p,r}} ~+~ \\
K_2\int_0^t  e^{K_2V_n(t)-K_2V_n(\tau)}\left(\|\rho^nu^{n}_x (\tau)\|_{B^{s-1}_{p,r}}+\|u^n_x(\tau)\|_{B^{s-1}_{p,r}}\right)~d\tau
\end{multline}
where 
\begin{equation}\label{Vn}
V_n (t) = \int_0^t \|u_x^{n}(\tau)\|_{B^{s-1}_{p,r}}~d \tau~\leq~ \int_0^t \|u^{n}(\tau)\|_{B^{s}_{p,r}}~d \tau~.
\end{equation}
\medskip

\noindent From \eqref{lambinv}, it follows that for some constant $M_1$, we have
\begin{equation}\label{est3}
\|\Lambda^{-1}\partial_x\left(\rho^n - u^n\right)\|_{B^s_{p,r}} \leq M_1\left(\|u^n\|_{B^s_{p,r}} + \|\rho^n\|_{B^{s-1}_{p,r}}\right)
\end{equation}
and by (6) in Lemma \ref{besov}, for some constant $M_2$, it holds that
\begin{equation}\label{est4}
\|\rho^nu^{n}_x\|_{B^{s-1}_{p,r}} \leq M_2\|u^n\|_{B^s_{p,r}}\|\rho^n\|_{B^{s-1}_{p,r}}~.
\end{equation}

\noindent Using \eqref{est3} in \eqref{est1}, \eqref{est4} in \eqref{est2} and setting $L_1 := \max\{K_1, M_1\}$, $L_2 := \max\{K_2, M_2\}$ we get
\begin{multline}\label{est5}
\|u^{n+1}(t)\|_{B^s_{p,r}}~\leq~ e^{L_1V_n(t)}\|u_0\|_{B^s_{p,r}} ~+~\\ L_1\int_0^t  e^{L_1V_n(t)-L_1V_n(\tau)}\left(\|u^n\|_{B^s_{p,r}} + \|\rho^n\|_{B^{s-1}_{p,r}}\right)~d\tau
\end{multline}
and 
\begin{align}\label{est6}
\begin{split}
\|\rho^{n+1}(t)\|_{B^{s-1}_{p,r}} & \leq~ e^{L_2V_n(t)}\|\rho_0\|_{B^{s-1}_{p,r}} ~+~L_2\int_0^t  e^{L_2V_n(t)-L_2V_n(\tau)}\|u^n\|_{B^s_{p,r}}\|\rho^n\|_{B^{s-1}_{p,r}}~d\tau\\ & + L_2\int_0^t  e^{L_2V_n(t)-L_2V_n(\tau)}\|u^n\|_{B^s_{p,r}}~d\tau\\
& \leq~ e^{L_2V_n(t)}\|\rho_0\|_{B^{s-1}_{p,r}} ~+~L_2\int_0^t  e^{L_2V_n(t)-L_2V_n(\tau)}\|u^n\|_{B^s_{p,r}}\|\rho^n\|_{B^{s-1}_{p,r}}~d\tau\\ & + L_2\int_0^t  e^{L_2V_n(t)-L_2V_n(\tau)}\left(\|u^n\|_{B^s_{p,r}} + \|\rho^n\|_{B^{s-1}_{p,r}}\right)~d\tau~.
\end{split}
\end{align}
Taking $C := 2\max\{L_1, L_2\}$ and using
\begin{equation*}
\|u^n\|_{B^s_{p,r}} \|\rho^n\|_{B^{s-1}_{p,r}} ~\leq~ \frac{\left(\|u^n\|_{B^s_{p,r}} + \|\rho^n\|_{B^{s-1}_{p,r}}\right)^2}{2}~,
\end{equation*}
we combine \eqref{est5} and \eqref{est6} to write
\begin{multline}\label{est7}
\|u^{n+1}(t)\|_{B^s_{p,r}} + \|\rho^{n+1}(t)\|_{B^{s-1}_{p,r}}~\leq~e^{CV_n(t)}\left(\|u_0\|_{B^s_{p,r}} + \|\rho_0\|_{B^{s-1}_{p,r}}\right)\\
 + C\int_0^t e^{CV_n(t) - CV_n(\tau)}\left(\|u^n\|_{B^s_{p,r}} + \|\rho^n\|_{B^{s-1}_{p,r}}\right)~d\tau\\
+ C\int_0^t e^{CV_n(t) - CV_n(\tau)} \frac{\left(\|u^n\|_{B^s_{p,r}} + \|\rho^n\|_{B^{s-1}_{p,r}}\right)^2}{2} ~d\tau~.
\end{multline}
Now we state and prove a lemma that provides the minimum lifespan. 
\begin{lemma}\label{lifespan}
There exists a minimum lifespan $T$ as stated in Theorem \ref{existmain}, such that for all $n\in \mathbb{N}$ and for all $t\in [0,T]$,
\begin{align}\label{ulife}
\|u^{n}(t)\|_{B^s_{p,r}} + \|\rho^{n}(t)\|_{B^{s-1}_{p,r}}&~\leq~ \frac{\|u_0\|_{B^s_{p,r}} + \|\rho_0\|_{B^{s-1}_{p,r}}}{\left(1- 4C\left(\|u_0\|_{B^s_{p,r}} + \|\rho_0\|_{B^{s-1}_{p,r}}\right)^2~t \right)^{\frac{1}{2}}} \nonumber \\& \leq \frac{\|u_0\|_{B^s_{p,r}} + \|\rho_0\|_{B^{s-1}_{p,r}}}{1- 4C\left(\|u_0\|_{B^s_{p,r}} + \|\rho_0\|_{B^{s-1}_{p,r}}\right)^2~t }
\end{align}
and
\begin{equation}\label{ulife1}
\|u^{n}(t)\|_{B^s_{p,r}} + \|\rho^{n}(t)\|_{B^{s-1}_{p,r}}~\leq~2\left(\|u_0\|_{B^s_{p,r}} + \|\rho_0\|_{B^{s-1}_{p,r}}\right)~.
\end{equation}
\end{lemma}
\noindent {\bf Proof of Lemma \ref{lifespan}:} We prove this result inductively. The cases $n=0$ and $n=1$ follow trivially. Set $P_0 = \|u_0\|_{B^s_{p,r}} + \|\rho_0\|_{B^{s-1}_{p,r}}$. For $n\in \mathbb{N}$, using \eqref{Vn} and \eqref{ulife} it follows that, for every $t\in[0,T]$
\begin{equation*}
V_n(t)~\leq~ -\frac{1}{4C}\ln\left(1- 4CP_0^2~t \right)~.
\end{equation*}
Then for every $t, \tau \in[0,T]$,
\begin{equation}\label{ev1}
e^{CV_n(t)}~\leq~ \left(1- 4CP_0^2~t \right)^{-\frac{1}{4}}
\end{equation}
which implies
\begin{equation}\label{ev2}
e^{CV_n(t)-CV_n(\tau)}~\leq~ \left(\frac{1- 4CP_0^2~\tau}{1- 4CP_0^2~t} \right)^{\frac{1}{4}}~.
\end{equation}
Substituting \eqref{ulife}, \eqref{ev1} and \eqref{ev2} in \eqref{est7} we get, 
\begin{equation}\label{comp}
\|u^{n+1}(t)\|_{B^s_{p,r}} + \|\rho^{n+1}(t)\|_{B^{s-1}_{p,r}}~\leq~\frac{P_0}{\left(1- 4CP_0^2~t \right)^{\frac{1}{4}}} + I_1 + I_2
\end{equation}
where
\begin{equation}\label{I1}
I_1 ~=~ \frac{C}{\left(1- 4CP_0^2~t \right)^{\frac{1}{4}}}\int_0^t \frac{P_0}{\left(1- 4CP_0^2~\tau \right)^{\frac{1}{4}}}~d\tau
\end{equation}
and 
\begin{equation}\label{I2}
I_2~=~\frac{C}{2\left(1- 4CP_0^2~t \right)^{\frac{1}{4}}}\int_0^t \frac{P_0^2}{\left(1- 4CP_0^2~\tau \right)^{\frac{3}{4}}}~d\tau~.
\end{equation}
Simplifying $I_1$, $I_2$ using the relation $\left(1- 4CP_0^2~t \right)^{\frac{1}{4}} \geq \left(1- 4CP_0^2~t \right)^{\frac{1}{2}}$ for any $t \in [0, T]$ and combining with \eqref{comp} yields
\begin{equation}
\|u^{n+1}(t)\|_{B^s_{p,r}} + \|\rho^{n+1}(t)\|_{B^{s-1}_{p,r}}~\leq~ \frac{P_0}{\left(1- 4CP_0^2~t \right)^{\frac{1}{2}}}~.
\end{equation}
This completes the proof of Lemma \ref{lifespan} by induction.\\
\quad\\
Next we show that $\{(u^n, \rho^n)\}_{n\geq0}$ converges to a solution $(u, \rho)$ of the system \eqref{FW}. We begin by using Arzela-Ascoli's theorem to find limit points $u\in C\left([0,T];B^{s-1}_{p,r}\right)$ and $\rho\in C\left([0,T];B^{s-2}_{p,r}\right)$ for $\{u^n\}_{n\geq0}$ and $\{\rho^n\}_{n\geq0}$ respectively.
\smallskip

\noindent   By Lemma \ref{lifespan}, $\{u^n\}_{n\geq0}$ is uniformly bounded in $C\left([0,T];B^{s}_{p,r}\right)$ and $\{ \rho^n\}_{n\geq0}$ is uniformly bounded in $C\left([0,T];B^{s-1}_{p,r}\right)$. Therefore, to apply Arzela-Ascoli's theorem, it is enough to show that $\{u^n\}_{n\geq0}$ is equicontinuous in $C\left([0,T];B^{s-1}_{p,r}\right)$ and $\{\rho^n\}_{n\geq 0}$ is equicontinuous in $C\left([0,T];B^{s-2}_{p,r}\right)$. Let $t_1, t_2 \in [0, T]$. By Mean Value theorem,
\begin{equation}\label{s2-1}
\|u^n(t_1) - u^n(t_2)\|_{B^{s-1}_{p,r}}~\leq~|t_1 - t_2|\sup\limits_{t\in[0,T]}\|u^n_t\|_{B^{s-1}_{p,r}}~.
\end{equation}
From \eqref{transport} we have
\begin{equation}
\|u^n_t\|_{B^{s-1}_{p,r}}~\leq~\|u^{n-1}u^{n}_x\|_{B^{s-1}_{p,r}} + \|\Lambda^{-1} \partial_x\left(\rho^{n-1} - u^{n-1}\right)\|_{B^{s-1}_{p,r}}~.
\end{equation}
As $B^{s-1}_{p,r}$ is an algebra, using \eqref{est3} the above becomes
\begin{equation}\label{s2-2}
\|u^n_t\|_{B^{s-1}_{p,r}}\leq~\|u^{n-1}\|_{B^{s-1}_{p,r}}\|u^{n}_x\|_{B^{s-1}_{p,r}} + M_1\left(\|u^{n-1}\|_{B^{s-1}_{p,r}} + \|\rho^{n-1}\|_{B^{s-2}_{p,r}}\right)~.
\end{equation}
Applying \eqref{ulife1} to \eqref{s2-2} and substituting the resulting inequality in \eqref{s2-1}, we get
\begin{equation}
\|u^n(t_1) - u^n(t_2)\|_{B^{s-1}_{p,r}}~\leq~\Gamma_1 \cdot |t_1 - t_2|~,
\end{equation}
where $\Gamma_1 = 2P_0(M_1 + 2P_0)$. Thus $\{u^n\}_{n\geq 0}$ is equicontinuous in $C\left([0,T];B^{s-1}_{p,r}\right)$ and converges to a limit $u$ in $C\left([0,T];B^{s-1}_{p,r}\right)$. Again, by Mean Value theorem,
\begin{equation}\label{s2-3}
\|\rho^n(t_1) - \rho^n(t_2)\|_{B^{s-2}_{p,r}}~\leq~|t_1 - t_2|\sup\limits_{t\in[0,T]}\|\rho^n_t\|_{B^{s-2}_{p,r}}~.
\end{equation}
Using \eqref{transport} we have
\begin{equation}
\|\rho^n_t\|_{B^{s-2}_{p,r}}~\leq~\|u^{n-1} \rho^{n}_x\|_{B^{s-2}_{p,r}} + \|\rho^{n-1}u^{n-1}_x\|_{B^{s-2}_{p,r}} + \|u^{n-1}_x\|_{B^{s-2}_{p,r}}
\end{equation}
and as $B^{s-2}_{p,r}$ is an algebra, from \eqref{est4} we get
\begin{equation}\label{s2-4}
\|\rho^n_t\|_{B^{s-2}_{p,r}}~\leq~\|u^{n-1}\|_{B^{s-2}_{p,r}}\|\rho^{n}_x\|_{B^{s-2}_{p,r}} + M_2\|u^{n-1}\|_{B^{s-1}_{p,r}}\|\rho^{n-1}\|_{B^{s-2}_{p,r}} + \|u^{n-1}\|_{B^{s-1}_{p,r}}~.
\end{equation}
Putting \eqref{ulife} in \eqref{s2-4} and substituting the result in \eqref{s2-3} yields
\begin{equation}
\|\rho^n(t_1) - \rho^n(t_2)\|_{B^{s-2}_{p,r}}~\leq~\Gamma_2\cdot |t_1 - t_2|~,
\end{equation}
where $\Gamma_2 = 2P_0[1+2(1+M_2)P_0]$. Hence $\{\rho^n\}_{n\geq 0}$ is equicontinuous in $C\left([0,T];B^{s-2}_{p,r}\right)$ and converges to a limit $\rho$ in $C\left([0,T];B^{s-2}_{p,r}\right)$. Using Cantor's diagonalization argument similar to \cite{D1}, for any test function $\varphi \in C_c^{\infty}(\R)$, as $n \to \infty$, $\|\varphi u_n (t) - \varphi u (t)\|_{B^{s-1}_{p,r}}$ and $\|\varphi \rho_n (t) - \varphi \rho (t)\|_{B^{s-2}_{p,r}}$ converge uniformly to $0$ on $[0,T]$. By the Fatou property of Besov spaces from Lemma \ref{besov}, for all $t \in [0,T]$  
$$\|u(t)\|_{B^{s}_{p,r}} \leq \kappa_1 \liminf\limits_{k\to\infty}\|u^{n}(t)\|_{B^{s}_{p,r}}~$$
and
$$\|\rho(t)\|_{B^{s-1}_{p,r}} \leq  \kappa_2 \liminf\limits_{k\to\infty}\|\rho^{n}(t)\|_{B^{s-1}_{p,r}}~$$
for some constants $\kappa_1$ and $\kappa_2$ that depend on $s$, $p$ and $r$. This implies $u \in {\bf L}^{\infty}\left([0,T];B^{s}_{p,r}\right)$ and $\rho \in {\bf L}^{\infty}\left([0,T];B^{s-1}_{p,r}\right)$. Now we show that $u \in C\left([0,T];B^{s}_{p,r}\right)$ and $\rho \in C\left([0,T];B^{s-1}_{p,r}\right)$. This implies verifying that for every $t \in (0, T)$,
\begin{equation}\label{ulim}
\lim_{\tau \to t} \|u(\tau) - u(t)\|_{B^{s}_{p,r}} = 0
\end{equation}
and
\begin{equation}\label{rholim}
\lim_{\tau \to t} \|\rho(\tau) - \rho(t)\|_{B^{s-1}_{p,r}} = 0~.
\end{equation} 
Let $\varepsilon >0$. To establish \eqref{ulim} we need to choose $\delta >0$ such that $\|u(\tau) - u(t)\|_{B^{s}_{p,r}} < \varepsilon$ whenever $|\tau -t|<\delta$. For any $n\in \mathbb{N}$, by triangle inequality
\begin{equation} \label{triangle1}
\|u(\tau) - u(t)\|_{B^{s}_{p,r}}~\leq~ \|u(\tau) - u^n(\tau)\|_{B^{s}_{p,r}} + \|u^n(\tau) - u^n(t)\|_{B^{s}_{p,r}} + \|u^n(t) - u(t)\|_{B^{s}_{p,r}}~.
\end{equation}
From the Fatou property in Lemma \ref{besov} it follows that $\{u^n\}_{n\geq 0}$ converges to $u$ in ${\bf L}^{\infty}\left([0,T];B^s_{p,r}\right)$, so there exists $N_0\in \mathbb{N}$ such that $\|u^n(\tau) - u(\tau)\|_{B^{s}_{p,r}} < \frac{\varepsilon}{3}$ and $\|u^n(t) - u(t)\|_{B^{s}_{p,r}} < \frac{\varepsilon}{3}$ for all $n \geq N_0$. Choosing $N>N_0$ sufficiently large, from \eqref{triangle1} we have
\begin{equation}\label{newtriangle1}
\|u(\tau) - u(t)\|_{B^{s}_{p,r}}~\leq~ \frac{2\varepsilon}{3} + \|u^{N}(\tau) - u^{N}(t)\|_{B^{s}_{p,r}}~.
\end{equation}
As $u^{N} \in C\left([0,T];B^{s}_{p,r}\right)$ by Lemma \ref{lifespan}, there exists $\delta>0$ depending on $N$ such that 
$\|u^{N}(\tau) - u^{N}(t)\|_{B^{s}_{p,r}} < \frac{\varepsilon}{3}$ whenever $|\tau -t|<\delta$. Therefore, \eqref{newtriangle1} yields \eqref{ulim}, and \eqref{rholim} follows by a similar argument.\\

\noindent Thus, we have $\left(u, \rho\right) \in C\left([0,T];B^{s}_{p,r} \times B^{s-1}_{p,r}\right)$, which proves the existence of a solution to the FW system \eqref{FW}.
\end{proof}

\subsection{Uniqueness} \label{subsec2}
The following proposition establishes uniqueness of the solution to \eqref{FW} by showing how a change in initial data affects the solution.\\

\begin{proposition}\label{uniqueness}
Let $s> \max\{2+\frac{1}{p}, \frac{5}{2}\}$, $p \in [1, \infty]$ and $r \in [1, \infty)$. Consider two solutions $(u^{(1)}, \rho^{(1)})$ and $(u^{(2)}, \rho^{(2)})$ of \eqref{FW} in $C\left([0,T];B^s_{p,r} \times B^{s-1}_{p,r}\right)$, corresponding to initial data $(u^{(1)}_{0},\rho^{(1)}_{0})$ and $(u^{(2)}_{0}, \rho^{(2)}_{0})$ respectively in $B^s_{p,r} \times B^{s-1}_{p,r}$. Let $w=u^{(1)} - u^{(2)}$, $v=\rho^{(1)} - \rho^{(2)}$, $w_0 = u^{(1)}_{0} - u^{(2)}_{0}$ and $v_0 = \rho^{(1)}_{0} - \rho^{(2)}_{0}$. Then for $\beta \in \mathbb{R}$ it holds that
\begin{equation}\label{uniq}
\|w(t)\|_{B^{s-1}_{p,r}} + \|v(t)\|_{B^{s-2}_{p,r}}~\leq~\left(\|w_0\|_{B^{s-1}_{p,r}} + \|v_0\|_{B^{s-2}_{p,r}}\right)e^{\beta t}~.
\end{equation}
\end{proposition}
\begin{proof}
First, we observe that $(w,v)$ satisfies the system of linear transport equations given by
\begin{equation}\label{uniqtransp}
\begin{cases}
\partial_t w + u^{(2)}\partial_x w = -w\partial_x u^{(1)} + \Lambda^{-1}[ \partial_x\left(v-w\right)]\\
\partial_t v + u^{(2)}\partial_x v = -w\partial_x \rho^{(1)} - v \partial_x u^{(1)} - \rho^{(2)}\partial_x w - \partial_x w
\end{cases}~.
\end{equation}
From Proposition \ref{danchinLT}, we obtain a linear transport estimate for each equation in system \eqref{uniqtransp}. Combining these estimates and using properties (4) and (5) from Lemma \ref{besov}, it holds that
\begin{multline}\label{unicalc1}
\|w(t)\|_{B^{s-1}_{p,r}} + \|v(t)\|_{B^{s-2}_{p,r}}~\leq~\left(\|w_0\|_{B^{s-1}_{p,r}} + \|v_0\|_{B^{s-2}_{p,r}}\right)e^{R\int_0^t \|\partial_x u^{(2)}(\tau')\|_{B^{s-2}_{p,r}}~d\tau'} + \\
R\int_0^te^{R\int_0^t \|\partial_x u^{(2)}(\tau')\|_{B^{s-2}_{p,r}}~d\tau'}\left(\|w(\tau)\|_{B^{s-1}_{p,r}} + \|v(\tau)\|_{B^{s-2}_{p,r}}\right)\cdot \\
\left(\|u^{(1)}(\tau)\|_{B^{s}_{p,r}}+\|\rho^{(1)}(\tau)\|_{B^{s-1}_{p,r}}+\|u^{(2)}(\tau)\|_{B^{s}_{p,r}}+\|\rho^{(2)}(\tau)\|_{B^{s-1}_{p,r}}\right)d\tau
\end{multline}
for some constant $R$ which depends on $s$, $p$ and $r$.
Differentiating \eqref{unicalc1} with respect to $t$, using the Fundamental Theorem of Calculus we get
\begin{multline}\label{groncalc}
\frac{d}{dt}\left[\left(\|w(t)\|_{B^{s-1}_{p,r}} + \|v(t)\|_{B^{s-2}_{p,r}}\right)e^{-R\int_0^t \|\partial_x u^{(2)}(\tau')\|_{B^{s-2}_{p,r}}~d\tau'}\right] \\
~\leq~ Re^{-R\int_0^t \|\partial_x u^{(2)}(\tau')\|_{B^{s-2}_{p,r}}~d\tau'}\left(\|w(t)\|_{B^{s-1}_{p,r}} + \|v(t)\|_{B^{s-2}_{p,r}}\right)\cdot \\
\left(\|u^{(1)}(t)\|_{B^{s}_{p,r}}+\|\rho^{(1)}(t)\|_{B^{s-1}_{p,r}}+\|u^{(2)}(t)\|_{B^{s}_{p,r}}+\|\rho^{(2)}(t)\|_{B^{s-1}_{p,r}}\right)~.
\end{multline}
Applying Gronwall's inequality on \eqref{groncalc} yields \eqref{uniq}, which proves that the solution to the FW system \eqref{FW} is unique.
\end{proof}

\subsection{Continuous dependence on initial data} \label{subsec3}

Consider initial data $(u_0, \rho_0) \in B^s_{p,r} \times B^{s-1}_{p,r}$ and let $(u, \rho)$ be its corresponding solution to the FW system \eqref{FW}. Let $\{{u_0}^j\}_{j\geq 0}$ converge to $u_0$ in $B^s_{p,r}$ and $\{{\rho_0}^j\}_{j\geq 0}$ converge to $\rho_0$ in $B^{s-1}_{p,r}$. Let solutions to \eqref{FW} corresponding to initial data $\{\left({u_0}^j, {\rho_0}^j \right)\}_{j\geq 0}$ be given by $\{\left({u}^j, {\rho}^j \right)\}_{j\geq 0}$. To establish continuity of the data-to-solution map, we need to prove that $\{\left({u}^j, {\rho}^j \right)\}_{j\geq 0}$ converges to $(u, \rho)$ in $C\left([0,T];B^{s}_{p,r} \times B^{s-1}_{p,r}\right)$, i.e.

\begin{equation}\label{ulimfin}
\lim_{j \to \infty} \|u^j - u\|_{C\left([0, T]; B^{s}_{p,r}\right)} = 0
\end{equation}
and
\begin{equation}\label{rholimfin}
\lim_{j \to \infty} \|\rho^j - \rho\|_{C\left([0, T]; B^{s-1}_{p,r}\right)} = 0~.
\end{equation} 
\\Let $\varepsilon > 0$. Suppose $\left({u_\varepsilon}^j, {\rho_\varepsilon}^j \right) $ denotes the solution to the FW system \eqref{FW} corresponding to mollified initial data $\left(J_{1/\varepsilon} {u_0}^j, J_{1/\varepsilon} {\rho_0}^j\right)$ and $\left({u_\varepsilon}, \rho_{\varepsilon}\right)$ denotes the solution to \eqref{FW} corresponding to mollified initial data $\left(J_{1/\varepsilon} {u_0}, J_{1/\varepsilon} {\rho_0}\right)$. By triangle inequality,
\begin{multline}\label{fineq1}
\|u^j - u\|_{C\left([0, T]; B^{s}_{p,r}\right)} \leq \|u^j - {u_\varepsilon}^j\|_{C\left([0, T]; B^{s}_{p,r}\right)} + \|{u_\varepsilon}^j - {u_\varepsilon}\|_{C\left([0, T]; B^{s}_{p,r}\right)}\\ + \|{u_\varepsilon} - u\|_{C\left([0, T]; B^{s}_{p,r}\right)}~.
\end{multline}
The first and the last terms on the right hand side of \eqref{fineq1} are identical in nature, therefore it is enough to estimate one of them. Without loss of generality, we find an estimate for the last term. Let $\left(u^n, \rho^n \right)$ be the approximate solution to the linear transport system \eqref{transport} corresponding to initial data $\left( J_n u_0, J_n \rho_0 \right)$. Then we have
\begin{equation}\label{fineq2}
\|{u_\varepsilon} - u\|_{C\left([0, T]; B^{s}_{p,r}\right)} \leq \|{u_\varepsilon} - u^n\|_{C\left([0, T]; B^{s}_{p,r}\right)} + \|u^n - u\|_{C\left([0, T]; B^{s}_{p,r}\right)}~.
\end{equation}
From subsection \ref{subsec1}, we have that $\lim_{n \to \infty} \|u^n - u\|_{C\left([0, T]; B^{s}_{p,r}\right)} = 0$. So there exists $P_1 \in \mathbb{N}$ such that $\|u^n - u\|_{C\left([0, T]; B^{s}_{p,r}\right)} < \frac{\varepsilon}{6}$ for all $n\geq P_1$.\\

\noindent Let $\left({u_\varepsilon}^n, {\rho_\varepsilon}^n \right) $ denote the approximate solution to system \eqref{transport} corresponding to mollified initial data $(J_n J_{1/\varepsilon} {u_0}, J_n J_{1/\varepsilon} {\rho_0})$. Then from the first term on the right hand side of \eqref{fineq2}, we get
\begin{equation}\label{fineq3}
\|{u_\varepsilon} - u^n\|_{C\left([0, T]; B^{s}_{p,r}\right)} \leq \|{u_\varepsilon} - {u_{\varepsilon}}^n\|_{C\left([0, T]; B^{s}_{p,r}\right)} + \|{u_{\varepsilon}}^n - u^n\|_{C\left([0, T]; B^{s}_{p,r}\right)}~.
\end{equation}
As subsection \ref{subsec1} shows that $\lim_{n \to \infty} \|{u_\varepsilon}^n - {u_\varepsilon}\|_{C\left([0, T]; B^{s}_{p,r}\right)} = 0$, there exists $P_2 \in \mathbb{N}$ such that $\|{u_\varepsilon}^n - {u_\varepsilon}\|_{C\left([0, T]; B^{s}_{p,r}\right)} < \frac{\varepsilon}{12}$ for all $n\geq P_2$.\\

\noindent Define ${w_\varepsilon}^n = {u_{\varepsilon}}^n - u^n$ and ${v_\varepsilon}^n = {\rho_{\varepsilon}}^n - \rho^n$. Then $\left({w_\varepsilon}^n, {v_\varepsilon}^n \right)$ is a solution to the linear transport system \eqref{transport} with initial data
$$\begin{cases}
{w_\varepsilon}^n (x,0) = J_n J_{1/\varepsilon} {u_0} - J_n u_0\\
{v_\varepsilon}^n (x,0) = J_n J_{1/\varepsilon} {\rho_0} - J_n \rho_0
\end{cases}~.$$
Choosing $1/\varepsilon$ sufficiently large and using the linear transport estimate from Proposition \ref{danchinLT}, we get
$\|{w_\varepsilon}^n\|_{C\left([0, T]; B^{s}_{p,r}\right)} \leq \|J_n J_{1/\varepsilon} {u_0} - J_n u_0\|_{C\left([0, T]; B^{s}_{p,r}\right)} < \frac{\varepsilon}{12}~.$
Therefore, it follows from \eqref{fineq3} that $\|{u_\varepsilon} - u^n\|_{C\left([0, T]; B^{s}_{p,r}\right)} < \frac{\varepsilon}{6}$ for all $n\geq P_2$. Take $P = \max\{P_1, P_2\}$. Then \eqref{fineq2} implies $\|{u_\varepsilon} - u\|_{C\left([0, T]; B^{s}_{p,r}\right)} < \frac{\varepsilon}{3}$ and \eqref{fineq1} yields that for all $j\geq P$,
\begin{equation}\label{fineq4}
\|u^j - u\|_{C\left([0, T]; B^{s}_{p,r}\right)} < \frac{\varepsilon}{3} +  \|{u_\varepsilon}^j - {u_\varepsilon}\|_{C\left([0, T]; B^{s}_{p,r}\right)} + \frac{\varepsilon}{3}~.
\end{equation}
Since the mollified initial data $\left(J_{1/\varepsilon} {u_0}^j, J_{1/\varepsilon} {\rho_0}^j\right)$ and $\left(J_{1/\varepsilon} {u_0}, J_{1/\varepsilon} {\rho_0}\right)$ are in $B^{s+1}_{p,r} \times B^{s}_{p,r}$, the corresponding solutions $\left({u_\varepsilon}^j, {\rho_\varepsilon}^j \right) $ and $\left({u_\varepsilon}, \rho_{\varepsilon}\right)$ belong to $C\left([0, T]; B^{s+1}_{p,r} \times B^{s}_{p,r} \right)$. Take ${w_\varepsilon}^j = {u_{\varepsilon}}^j - u_{\varepsilon}$ and ${v_\varepsilon}^j = {\rho_{\varepsilon}}^j - \rho_{\varepsilon}$. Then $\left(w_{\varepsilon}^j, v_{\varepsilon}^j  \right)$ satisfies the system of linear transport equations given by
\begin{equation}\label{uniqtrans}
\begin{cases}
\partial_t {w_\varepsilon}^j + {u_\varepsilon}\partial_x {w_\varepsilon}^j = -{w_\varepsilon}^j\partial_x {u_\varepsilon}^j + \Lambda^{-1}[ \partial_x\left({v_\varepsilon}^j - {w_\varepsilon}^j \right)]\\
\partial_t {v_\varepsilon}^j + {u_\varepsilon}\partial_x {v_\varepsilon}^j = -{w_\varepsilon}^j \partial_x {\rho_\varepsilon}^j - {v_\varepsilon}^j \partial_x {u_\varepsilon}^j - {\rho_\varepsilon} \partial_x {w_\varepsilon}^j - \partial_x {w_\varepsilon}^j
\end{cases}~.
\end{equation}
Applying Proposition \ref{danchinLT} on the first equation in system \eqref{uniqtrans}, we obtain 
\begin{equation}
\|{u_{\varepsilon}}^j - u_{\varepsilon}\|_{B^{s}_{p,r}} \leq \|{u_0}^j - {u_0}\|_{B^{s}_{p,r}}~.
\end{equation}
As $\{{u_0}^j\}_{j\geq 0}$ converges to $u_0$, there exists $p \in \mathbb{N}$ such that $\|{u_{\varepsilon}}^j - u_{\varepsilon}\|_{B^{s}_{p,r}} < \frac{\varepsilon}{3}$ for all $j \geq p$. Take $p_0 = \max\{P, p \}$. Then \eqref{fineq4} implies that for every $j\geq p_0$, $$\|u^j - u\|_{C\left([0, T]; B^{s}_{p,r}\right)} < \varepsilon~,$$ which proves \eqref{ulimfin}. Similarly, we have
\begin{multline}\label{fineq5}
\|\rho^j - \rho\|_{C\left([0, T]; B^{s-1}_{p,r}\right)} \leq \|\rho^j - {\rho_\varepsilon}^j\|_{C\left([0, T]; B^{s-1}_{p,r}\right)} + \|{\rho_\varepsilon}^j - {\rho_\varepsilon}\|_{C\left([0, T]; B^{s-1}_{p,r}\right)}\\ + \|{\rho_\varepsilon} - \rho\|_{C\left([0, T]; B^{s-1}_{p,r}\right)}~.
\end{multline}
To estimate the last term on the right hand side, triangle inequality implies
\begin{equation}\label{fineq6}
\|{\rho_\varepsilon} - \rho\|_{C\left([0, T]; B^{s-1}_{p,r}\right)} \leq \|{\rho_\varepsilon} - \rho^n\|_{C\left([0, T]; B^{s-1}_{p,r}\right)} + \|\rho^n - \rho\|_{C\left([0, T]; B^{s-1}_{p,r}\right)}~.
\end{equation}
Again $\lim_{n \to \infty} \|\rho^n - \rho\|_{C\left([0, T]; B^{s-1}_{p,r}\right)} = 0$ by subsection \ref{subsec1}, so there exists $Q_1 \in \mathbb{N}$ such that $\|\rho^n - \rho\|_{C\left([0, T]; B^{s-1}_{p,r}\right)} < \frac{\varepsilon}{6}$ for all $n\geq Q_1$. The first term on the right hand side of \eqref{fineq6} yields
\begin{equation}\label{fineq7}
\|{\rho_\varepsilon} - \rho^n\|_{C\left([0, T]; B^{s-1}_{p,r}\right)} \leq \|{\rho_\varepsilon} - {\rho_{\varepsilon}}^n\|_{C\left([0, T]; B^{s-1}_{p,r}\right)} + \|{\rho_{\varepsilon}}^n - \rho^n\|_{C\left([0, T]; B^{s-1}_{p,r}\right)}~.
\end{equation}
Following a procedure similar to that for \eqref{fineq3}, there exists $Q_2 \in \mathbb{N}$ such that $\|{\rho_\varepsilon}^n - {\rho_\varepsilon}\|_{C\left([0, T]; B^{s-1}_{p,r}\right)} < \frac{\varepsilon}{12}$ for all $n\geq Q_2$.\\

\noindent Recall that ${w_\varepsilon}^n = {u_{\varepsilon}}^n - u^n$ and ${v_\varepsilon}^n = {\rho_{\varepsilon}}^n - \rho^n$. As $\left({w_\varepsilon}^n, {v_\varepsilon}^n \right)$ is a solution to system \eqref{transport} with initial data
$$\begin{cases}
{w_\varepsilon}^n (x,0) = J_n J_{1/\varepsilon} {u_0} - J_n u_0\\
{v_\varepsilon}^n (x,0) = J_n J_{1/\varepsilon} {\rho_0} - J_n \rho_0
\end{cases}~,$$
we use the linear transport estimate from Proposition \ref{danchinLT} again and choose $1/\varepsilon$ sufficiently large, to get
$$\|{v_\varepsilon}^n\|_{C\left([0, T]; B^{s-1}_{p,r}\right)} \leq \|J_n J_{1/\varepsilon} {\rho_0} - J_n \rho_0\|_{C\left([0, T]; B^{s-1}_{p,r}\right)} < \frac{\varepsilon}{12}~.$$
Replacing this in \eqref{fineq7} yields that $\|{\rho_\varepsilon} - \rho^n\|_{C\left([0, T]; B^{s-1}_{p,r}\right)} < \frac{\varepsilon}{6}$ for all $n\geq Q_2$. Set $Q = \max\{Q_1, Q_2\}$. Then we have $\|{\rho_\varepsilon} - \rho\|_{C\left([0, T]; B^{s-1}_{p,r}\right)} < \frac{\varepsilon}{3}$ from \eqref{fineq6}. Consequently, \eqref{fineq5} implies that for all $j\geq Q$,
\begin{equation}\label{fineq8}
\|\rho^j - \rho\|_{C\left([0, T]; B^{s-1}_{p,r}\right)} < \frac{\varepsilon}{3} +  \|{\rho_\varepsilon}^j - {\rho_\varepsilon}\|_{C\left([0, T]; B^{s-1}_{p,r}\right)} + \frac{\varepsilon}{3}~.
\end{equation}
Now, using Proposition \ref{danchinLT} for the second equation in \eqref{uniqtrans} we obtain 
\begin{equation}
\|{\rho_{\varepsilon}}^j - \rho_{\varepsilon}\|_{B^{s-1}_{p,r}} \leq \|{\rho_0}^j - {\rho_0}\|_{B^{s-1}_{p,r}}~.
\end{equation}
Since $\{{\rho_0}^j\}_{j\geq 0}$ converges to $\rho_0$, there exists $q \in \mathbb{N}$ such that $\|{\rho_{\varepsilon}}^j - \rho_{\varepsilon}\|_{B^{s-1}_{p,r}} < \frac{\varepsilon}{3}$ for all $j \geq q$. Let $q_0 = \max\{Q, q \}$. Then from \eqref{fineq8}, it follows that for every $j\geq q_0$, $$\|\rho^j - \rho\|_{C\left([0, T]; B^{s-1}_{p,r}\right)} < \varepsilon~,$$ thus proving \eqref{rholimfin}.
\bigskip

\noindent This completes the proof of local well-posedness for the two-component Fornberg-Whitham system \eqref{FW} in Besov spaces $B^s_{p,r} \times B^{s-1}_{p,r}$ for $s > \max\{2+\frac{1}{p}, \frac{5}{2}\}$.
\vspace{0.3cm}
\subsection*{Acknowledgments}

I thank Barbara Lee Keyfitz and John Holmes, from The Ohio State University, for their valuable suggestions on this project. 



\end{document}